\documentclass[11pt]{amsart}
\author{Yakov Berchenko-Kogan}
\title[Numerically computing the index of self-shrinkers]{Numerically computing the index of mean curvature flow self-shrinkers}

\newcommand\res{med}

\usepackage[pdftitle={Numerically computing the index of mean curvature flow self-shrinkers}, unicode]{hyperref}
\usepackage[utf8]{inputenc}
\usepackage{graphicx, multirow, makecell,siunitx}

\newtheorem{theorem}{Theorem}[section]
\newtheorem{result}[theorem]{Result}
\newtheorem{proposition}[theorem]{Proposition}

\theoremstyle{definition}

\newtheorem{definition}[theorem]{Definition}

\newcommand\abs[1]{\left\lvert{#1}\right\rvert}
\newcommand\norm[1]{\left\lVert{#1}\right\rVert}
\newcommand\dd[2][]{\frac{d{#1}}{d{#2}}}
\newcommand\ddtwo[2][]{\frac{d^2{#1}}{d{#2}^2}}
\newcommand\R{\mathbb R}
\renewcommand\L{\mathcal L}
\newcommand\n{\mathbf n}
\newcommand\mbf\mathbf

\let\div\relax
\DeclareMathOperator\div{div}
\DeclareMathOperator\tr{tr}
\DeclareMathOperator\grad{grad}
\DeclareMathOperator\dist{dist}
\DeclareMathOperator\vol{area}

\begin{document}
\begin{abstract}
  Surfaces that evolve by mean curvature flow develop singularities. These singularities can be modeled by self-shrinkers, surfaces that shrink by dilations under the flow. Singularities modeled on classical self-shrinkers, namely spheres and cylinders, are stable under perturbations of the flow. In contrast, singularities modeled on other self-shrinkers, such as the Angenent torus, are unstable: perturbing the flow will generally change the kind of singularity. One can measure the degree of instability by computing the Morse index of the self-shrinker, viewed as a critical point of an appropriate functional.
  
  In this paper, we present a numerical method for computing the index of rotationally symmetric self-shrinkers. We apply this method to the Angenent torus, the first known nontrivial example of a self-shrinker. We find that, excluding dilations and translations, the index of the Angenent torus is $5$, which is consistent with the lower bound of $3$ from the work of Liu and the upper bound of $29$ from our earlier work. Also, we unexpectedly discover two additional variations of the Angenent torus with eigenvalue $-1$.
\end{abstract}

\maketitle

\section{Introduction}
Mean curvature flow is a geometric evolution equation for surfaces $\Sigma\subset\R^3$, under which each point $x\in\Sigma$ moves in the inward normal direction proportionally to the mean curvature of $\Sigma$ at $x$. More generally, mean curvature flow can be defined as the gradient flow for the area functional. This flow has many applications in geometry and topology, as well as in image denoising.

Under mean curvature flow, spheres and cylinders will shrink by dilations. However, these are not the only examples of such \emph{self-shrinkers}. Angenent found a self-shrinking torus in 1989 \cite{a92}, and since then many other examples have been constructed \cite{dk17, dln18, kkm18, km14, m15, mo11, n09, n10, n14}. These surfaces are important because they model singularities that develop under mean curvature flow. For example, if the initial surface is convex, then it will become rounder as it shrinks, eventually becoming close to a sphere before it shrinks to a point. On the other hand, if the initial surface is not convex, then it may develop a singularity that looks like a different self-shrinker.

Colding and Minicozzi \cite{cm12} show that spheres and cylinders are stable, in the sense that if we perturb a sphere or a cylinder, then the resulting mean curvature flow will still develop a spherical or cylindrical singularity, respectively. Other self-shrinkers are unstable. Nonetheless, the space of unstable variations, defined appropriately, is still finite-dimensional. As in Morse theory, this dimension is called the \emph{index} of the self-shrinker. Note, however, that there are different conventions for the index in the literature; the disagreement is about whether or not to include translations and dilations, which give rise to singularities of the same shape but at different places and times. We exclude translations and dilations in this work.

\subsection{Results}
In this paper, we compute the index of the Angenent torus. However, the computational methods in this paper apply equally well to any rotationally symmetric immersed self-shrinker. See \cite{dk17} for infinitely many examples of such self-shrinkers. More generally, we expect that combining the methods in this paper with finite element methods could be used to compute the indices of self-shrinkers that do not have rotational symmetry.

\begin{result}
  Excluding dilations and translations, the index of the Angenent torus is $5$.
\end{result}

Computing the index of a self-shrinker amounts to counting the number of negative eigenvalues of a certain differential operator called the \emph{stability operator}, which acts on functions that represent normal variations of the self-shrinker. In this paper, we construct a suitable finite-dimensional approximation to the stability operator, and then we compute the eigenvalues and eigenvectors of this matrix. We recover the facts \cite{cm12} that the variation corresponding to dilation has eigenvalue $-1$ and that the variations corresponding to translations have eigenvalue $-\frac12$. Additionally, we discover that, surprisingly, two other variations also have eigenvalue $-1$, and these variations have simple explicit formulas.

\subsection{Relationship to other work}
One can view this paper in several ways.
\begin{itemize}
\item This paper can be viewed a sequel to \cite{bk19}, where we compute the Angenent torus and its entropy. These results are the starting point for the index computation in the current work.
\item This paper can be viewed as a numerical implementation of \cite{l16}, where Liu gives a formula for the stability operator of rotationally symmetric self-shrinkers and shows, in particular, that the index of the Angenent torus is at least $3$.
\item This paper can be viewed as a numerical companion to \cite{bk20a}, where we prove upper bounds on the index of self-shrinking tori. The numerical discovery of the two new variations with eigenvalue $-1$ inspired a simple new formula for the stability operator \cite[Theorem 3.7]{bk20a}. In turn, this formula gives a two-line proof that these numerically discovered variations do indeed have eigenvalue $-1$ \cite[Theorem 6.1]{bk20a}.
\end{itemize}

\subsection{Outline}
In Section~\ref{sec:preliminaries}, we introduce notation and give basic properties of mean curvature flow, self-shrinkers, and the stability operator, both in general and in the rotationally symmetric case. In Section~\ref{sec:methods}, we discuss the numerical methods we use to compute the eigenvalues and eigenfunctions of the stability operator. We present our results in Section~\ref{sec:results}, listing the first several eigenvalues of the stability operator and presenting plots of the corresponding variations of the Angenent torus cross-section. We obtain the index by counting the variations with negative eigenvalues; we present three-dimensional plots of all of these variations in Figure~\ref{fig:3d}. Next, in Section~\ref{sec:error}, we estimate the error in our numerical computations by giving plots that show the rate at which our numerically computed eigenvalues converge as we increase the number of sample points. We summarize the first few computed eigenvalues and our error estimates in Table~\ref{tab:errorplots}. Finally, in Section~\ref{sec:future}, we give a few promising directions for future work. Additionally, we include Appendix~\ref{sec:asymptotics}, where, rather than looking at small eigenvalues, we instead take a cursory look at eigenvalue asymptotics.

\section{Preliminaries: Mean curvature flow and self-shrinkers}\label{sec:preliminaries}
In this section, we introduce mean curvature flow for surfaces in $\R^3$, define self-shrinkers and their index, and discuss the rotationally symmetric case. For a more detailed introduction in  general dimension, see the companion paper \cite{bk20a}. We also refer the reader to \cite{cm12, cmp15, h90, l16}.

\subsection{Notation for surfaces and mean curvature flow}
Let $\Sigma\subset\R^3$ be an immersed oriented surface. Let $\n$ denote the unit normal vector to $\Sigma$. Given a point $x\in\Sigma$ and a vector $v\in T_x\R^{n+1}$, let $v^\perp$ denote the scalar projection of $v$ onto $\n$, namely $v^\perp=\langle v,\n\rangle$. Let $v^\top$ denote the projection of $v$ onto $T_x\Sigma$, namely $v^\top=v-v^\perp\n$.

\begin{definition}
  Let $A_\Sigma$ denote the \emph{second fundamental form} of $\Sigma$. That is, given $v,w\in T_x\Sigma$, let
  \begin{equation*}
    A_\Sigma(v,w)=(\nabla_vw)^\perp
  \end{equation*}
\end{definition}

\begin{definition}
  Let $H_\Sigma$ denote the \emph{mean curvature} of $\Sigma$, defined with the normalization convention $H_\Sigma=-\tr A_\Sigma$.
\end{definition}

That is, if $e_1,e_2$ is an orthonormal frame at a particular point $x\in\Sigma$, then, at that point $x$, $H_\Sigma=-A_\Sigma(e_1,e_1)-A_\Sigma(e_2,e_2)$.

\begin{definition}
  A family of surfaces $\Sigma_t$ evolves under \emph{mean curvature flow} if
  \begin{equation*}
    \dot x=-H_\Sigma\n.
  \end{equation*}
  That is, each point on $\Sigma$ moves with speed $H_\Sigma$ in the inward normal direction.
\end{definition}

\subsection{Self-shrinkers}
A surface $\Sigma$ is a self-shrinker if it evolves under mean curvature flow by dilations. For this paper, however, we will restrict this terminology to refer only to surfaces that shrink to the origin in one unit of time.

\begin{definition}
  A surface $\Sigma$ is a \emph{self-shrinker} if $\Sigma_t=\sqrt{-t}\,\Sigma$ is a mean curvature flow for $t<0$.
\end{definition}

We have an extremely useful variational formulation for self-shrinkers as critical points of Huisken's $F$-functional.

\begin{definition}
  The \emph{$F$-functional} takes a surface and computes its weighted area via the formula
  \begin{equation*}
    F(\Sigma)=\frac1{4\pi}\int_\Sigma e^{-\abs x^2/4}\,d\vol_\Sigma.
  \end{equation*}
\end{definition}

The role of the normalization constant $\frac1{4\pi}$ is to ensure that if $\Sigma$ is a plane through the origin, then $F(\Sigma)=1$.

Since varying a surface in a tangential direction does not change the surface, we can define the critical points of $F$ solely in terms of normal variations.
\begin{definition}
  $\Sigma$ is a \emph{critical point} of $F$ if for any $f\colon\Sigma\to\R$ with compact support, $F$ does not change to first order as we vary $\Sigma$ by $f$ in the normal direction. More precisely, if we let $\Sigma_s=\{x+sf\n\mid x\in\Sigma\}$, then we have $\dd s\bigr\rvert_{s=0}F(\Sigma_s)=0$.
\end{definition}

\begin{proposition}[\cite{cm12}]\label{prop:shrinkercritical}
  $\Sigma$ is a self-shrinker if and only if $\Sigma$ is a critical point of $F$.
\end{proposition}

The definition of the $F$-functional is not invariant under translation and dilation: The Gaussian weight is centered around the origin in $\R^3$, and the length scale of the Gaussian is designed so that the critical surfaces become extinct in exactly one unit of time. Colding and Minicozzi introduce a related concept called the entropy, which coincides with the $F$-functional on self-shrinkers but is invariant under translation and dilation.
\begin{definition}
  The \emph{entropy} of a surface $\Sigma\subset\R^3$ is the supremum of the $F$-functional evaluated on all translates and dilates of $\Sigma$, that is $\sup_{x_0,t_0}F(x_0+\sqrt{t_0}\Sigma)$.
\end{definition}
If $\Sigma$ is a self-shrinker, defined as above to shrink to the origin in one unit of time, then the supremum among translates and dilates is attained at $\Sigma$ itself, so the entropy of $\Sigma$ coincides with $F(\Sigma)$. However, entropy-decreasing variations of $\Sigma$ and $F$-decreasing variations of $\Sigma$ are not quite the same: when we ask about entropy-decreasing variations, we exclude the ``trivial'' $F$-decreasing variations of translation and dilation.

\subsection{The stability operator}
Given a critical point of a flow, the next natural question to ask is about the stability of that critical point. If we perturb a self-shrinker, will the resulting surface flow back to the self-shrinker under the gradient flow for the $F$-functional, or will it flow to a different critical point? What is the maximum dimension of a space of unstable variations? As in Morse theory, answering this question amounts to computing the eigenvalues of the Hessian of the $F$-functional.

\begin{definition}\label{def:stability}
  Let $\Sigma$ be a self-shrinker. The \emph{stability operator} $L_\Sigma$ is a differential operator acting on functions $f\colon\Sigma\to\R$ that is the Hessian of $F$ in the following sense.
  \begin{itemize}
  \item For any $f\colon\Sigma\to\R$,
    \begin{equation}\label{eq:stability}
      \ddtwo s\Bigr\rvert_{s=0}F(\Sigma_s)=\frac1{4\pi}\int_\Sigma f(-L_\Sigma)f\,e^{-\abs x^2/4}\,d\vol_\Sigma,
    \end{equation}
    where $\Sigma_s=\{x+sf\n\mid x\in\Sigma\}$ is the normal variation corresponding to $f$, and
  \item $L_\Sigma$ is symmetric with respect to the Gaussian weight, in the sense that $\int_\Sigma f_1L_\Sigma f_2\,e^{-\abs x^2/4}\,d\vol_\Sigma=\int_\Sigma f_2L_\Sigma f_1\,e^{-\abs x^2/4}\,d\vol_\Sigma$.
  \end{itemize}
\end{definition}

The reason for the odd choice of sign is so that the differential operator $L_\Sigma$ has the same leading terms as the Laplacian $\Delta_\Sigma=\div_\Sigma\grad_\Sigma$. More precisely, Colding and Minicozzi \cite{cm12} compute that,
\begin{equation*}
  L_\Sigma=\L_\Sigma+\abs{A_\Sigma}^2+\tfrac12,
\end{equation*}
where
\begin{equation*}
  \L_\Sigma f=e^{\abs x^2/4}\div_\Sigma\left(e^{-\abs x^2/4}\grad_\Sigma f\right).
\end{equation*}
Consequently, we make the following sign convention for eigenvalues.

\begin{definition}[Sign convention for eigenvalues]
  We say that $f\neq0$ is an eigenfunction of a differential operator $L$ with eigenvalue $\lambda$ if $-Lf = \lambda f$.
\end{definition}

We conclude that eigenfunctions of the stability operator $L_\Sigma$ with negative eigenvalues are unstable variations of the self-shrinker $\Sigma$: If we vary $\Sigma$ in that direction, then the gradient flow for $F$ will take the surface away from $\Sigma$. Meanwhile, eigenfunctions of the stability operator $L_\Sigma$ with positive eigenvalues are stable variations of $\Sigma$: There exists a gradient flow line that approaches $\Sigma$ from that direction.

Translating $\Sigma$ in the direction $v\in\R^3$ corresponds to the normal variation $f=v^\perp$, and dilating $\Sigma$ corresponds to the normal variation $f=H_\Sigma$. Colding and Minicozzi compute the corresponding eigenvalues of the stability operator.
\begin{proposition}[\cite{cm12}]
  For any vector $v\in\R^3$, we have $L_\Sigma v^\perp=\frac12v^\perp$. Meanwhile, for dilation, we have $L_\Sigma H_\Sigma=H_\Sigma$.
\end{proposition}
Thus, assuming these functions are nonzero, $v^\perp$ and $H_\Sigma$ are eigenfunctions of $L_\Sigma$, giving us $4$ independent eigenfunctions. With our sign convention, the eigenvalue corresponding to $v^\perp$ is $-\frac12$, and the eigenvalue corresponding to $H_\Sigma$ is $-1$. Additionally, because $F$ is invariant under rotations about the origin, a variation of $\Sigma$ corresponding to a rotation about the origin will be an eigenfunction with eigenvalue $0$.

Because $L_\Sigma$ has the same symbol as $\Delta_\Sigma$, it has a finite number of negative eigenvalues, at least in the case of compact $\Sigma$. Usually, one defines the index of a critical point of a gradient flow to be the number of negative eigenvalues of the Hessian. However, because translations and dilations do not change the shape of the self-shrinker, we exclude them in this context.

\begin{definition}
  The \emph{index} of a self-shrinker $\Sigma$ is the number of negative eigenvalues of the stability operator $L_\Sigma$, excluding those eigenvalues corresponding to translations and dilations.
\end{definition}

Assuming that $\Sigma$ is not invariant under any translations or dilations, its index is simply $4$ less than the usual Morse index.

Under mild assumptions, Colding and Minicozzi show that the only self-shrinkers with index zero are planes, round spheres, and round cylinders \cite{cm12}

\subsection{Rotationally symmetric self-shrinkers}
If $\Sigma\subset\R^3$ is a hypersurface with $SO(2)$ rotational symmetry, we can understand it in terms of its cross-sectional curve $\Gamma$. We refer the reader to \cite{l16}.

We will use cylindrical coordinates $(r,\theta,z)$ on $\R^3$.

\begin{definition}
  We say that a hypersurface is \emph{rotationally symmetric} if it is invariant under rotations about the $z$-axis.
\end{definition}

If $\Sigma$ is rotationally symmetric, we let $\Gamma$ denote its $\theta=0$ cross-section, which we also think of as being a curve in the half-plane $\{(r,z)\mid r\ge0,z\in\R\}$.

We can write the $F$-functional in terms of $\Gamma$.
\begin{proposition}\label{prop:FGamma}
  If $\Sigma$ is a rotationally symmetric hypersurface with cross-section $\Gamma$, then
  \begin{equation*}
    F(\Sigma)=\frac12\int_\Gamma re^{-\abs x^2/4}\,d\ell,
  \end{equation*}
  where $d\ell$ denotes integration with respect to arc length.
\end{proposition}

To simplify our notation, we will let $\sigma$ denote this weight.
\begin{definition}
  Let $\sigma\colon\R_{\ge0}\times\R\to\R$ denote the weight
  \begin{equation*}
    \sigma=\frac12re^{-\abs x^2/4}.
  \end{equation*}
\end{definition}

With this notation, our expression $\frac12\int_\Gamma re^{-\abs x^2/4}\,d\ell$ for $F(\Sigma)$ is simply the length of the curve $\Gamma$ with respect to the conformally changed metric $\sigma^2(dr^2+dz^2)$.
\begin{definition}
  Let $g^\sigma$ denote the metric $\sigma^2(dr^2+dz^2)$ on the half-plane $\{(r,z)\mid r\ge0,z\in\R\}$.
\end{definition}
If $\Sigma$ is a self-shrinker, then $\Sigma$ is a critical point for $F$, so $\Gamma$ is a critical point for $g^\sigma$-length. In other words, $\Gamma$ is a geodesic with respect to $g^\sigma$.

\subsection{The stability operator for rotationally symmetric self-shrinkers}
Varying the cross-section $\Gamma$ only yields rotationally symmetric variations of $\Sigma$. To understand the stability operator $L_\Sigma$ in terms of $\Gamma$, we must understand non-rotationally symmetric variations of $\Sigma$ as well. Liu \cite{l16} does so by decomposing normal variations $f\colon\Sigma\to\R$ into their Fourier components. The stability operator $L_\Sigma$ commutes with this Fourier decomposition, so we can decompose $L_\Sigma$ into its Fourier components $L_k$, which are operators acting on functions on $\Gamma$.

We begin with the rotationally symmetric part of the Fourier decomposition.
\begin{definition}
  Let $\Sigma$ be a rotationally symmetric self-shrinker with cross-section $\Gamma$. For any $u\colon\Gamma\to\R$, we have a corresponding rotationally symmetric function $f\colon\Sigma\to\R$. Define the operator $L_0$ by
  \begin{equation*}
    L_0u=L_\Sigma f.
  \end{equation*}
\end{definition}

Note that, if $f$ is a rotationally symmetric function as above, then the normal variation $\Sigma_s=\{x+sf\n\mid x\in\Sigma\}$ has cross-section $\Gamma_s=\{x+su\n\mid x\in\Gamma\}$, and equation~\eqref{eq:stability} in Definition~\ref{def:stability} becomes
\begin{equation}\label{eq:l0}
  \ddtwo s\Bigr\rvert_{s=0}\int_{\Gamma_s}\sigma\,d\ell=\int_\Gamma u(-L_0)u\,\sigma\,d\ell,
\end{equation}
where $d\ell$ denotes integration with respect to the usual Euclidean arc length.

We now define the other Fourier components of $L_\Sigma$.
\begin{definition}
  For each non-negative integer $k$, let $L_k$ be the operator
  \begin{equation*}
    L_k=L_0-\frac{k^2}{r^2}.
  \end{equation*}
\end{definition}

\begin{proposition}[\cite{l16}]
  For any $u\colon\Gamma\to\R$ and any $k\ge0$, we have
  \begin{align*}
    L_\Sigma(u\cos k\theta)&=(L_ku)\cos k\theta,\\
    L_\Sigma(u\sin k\theta)&=(L_ku)\sin k\theta.
  \end{align*}
\end{proposition}
Thus, the eigenfunctions of $L_\Sigma$ are of the form $u\cos k\theta$ and $u\sin k\theta$, where $u$ is an eigenfunction of $L_k$. Consequently, we can determine the eigenvalues and eigenfunctions of the stability operator $L_\Sigma$ by determining the eigenvalues and eigenfunctions of $L_k$ for all $k\ge0$.

In the case where $\Sigma$ is a rotationally symmetric torus, we have the following formula for $L_k$.
\begin{proposition}[\cite{bk20a}]\label{prop:lk}
  Let $\Sigma$ be a rotationally symmetric torus with cross-section $\Gamma$. Then
  \begin{equation*}
    L_k=\sigma\Delta_\Gamma^\sigma\sigma+1+\frac{1-k^2}{r^2},
  \end{equation*}
  where $\sigma$ is the weight $\frac12re^{-\abs x^2/4}$ and $\Delta_\Gamma^\sigma$ is the Laplacian on $\Gamma$ with respect to the conformally changed metric $g^\sigma=\sigma^2(dr^2+dz^2)$.
\end{proposition}
See also Appendix~\ref{sec:asymptotics}, where we rewrite $L_k$ as a Schr\"odinger operator, and \cite{l16}, where Liu gives the general formula for $L_k$. 

\section{Numerical methods}\label{sec:methods}
In \cite{bk19}, we computed the Angenent torus cross-section $\Gamma$ by viewing it as a geodesic with respect to the metric $g^\sigma=\sigma^2(dr^2+dz^2)$. The result is a discrete approximation $\Gamma_d=\{q_0,q_1\dotsc,q_M=q_0\}$ to the curve $\Gamma$, where the $q_m$ are equally spaced with respect to the metric $g^\sigma$, as illustrated in Figure \ref{fig:curvedots}. As we will see, having the points $q_m$ be equally spaced in this way is particularly well-suited for computing the index of the Angenent torus. Additionally, we computed the entropy $F(\Sigma)$ of the Angenent torus, which is the length of the curve $\Gamma$ with respect to the metric $g^\sigma$.

\begin{figure}
  \centering
  \includegraphics{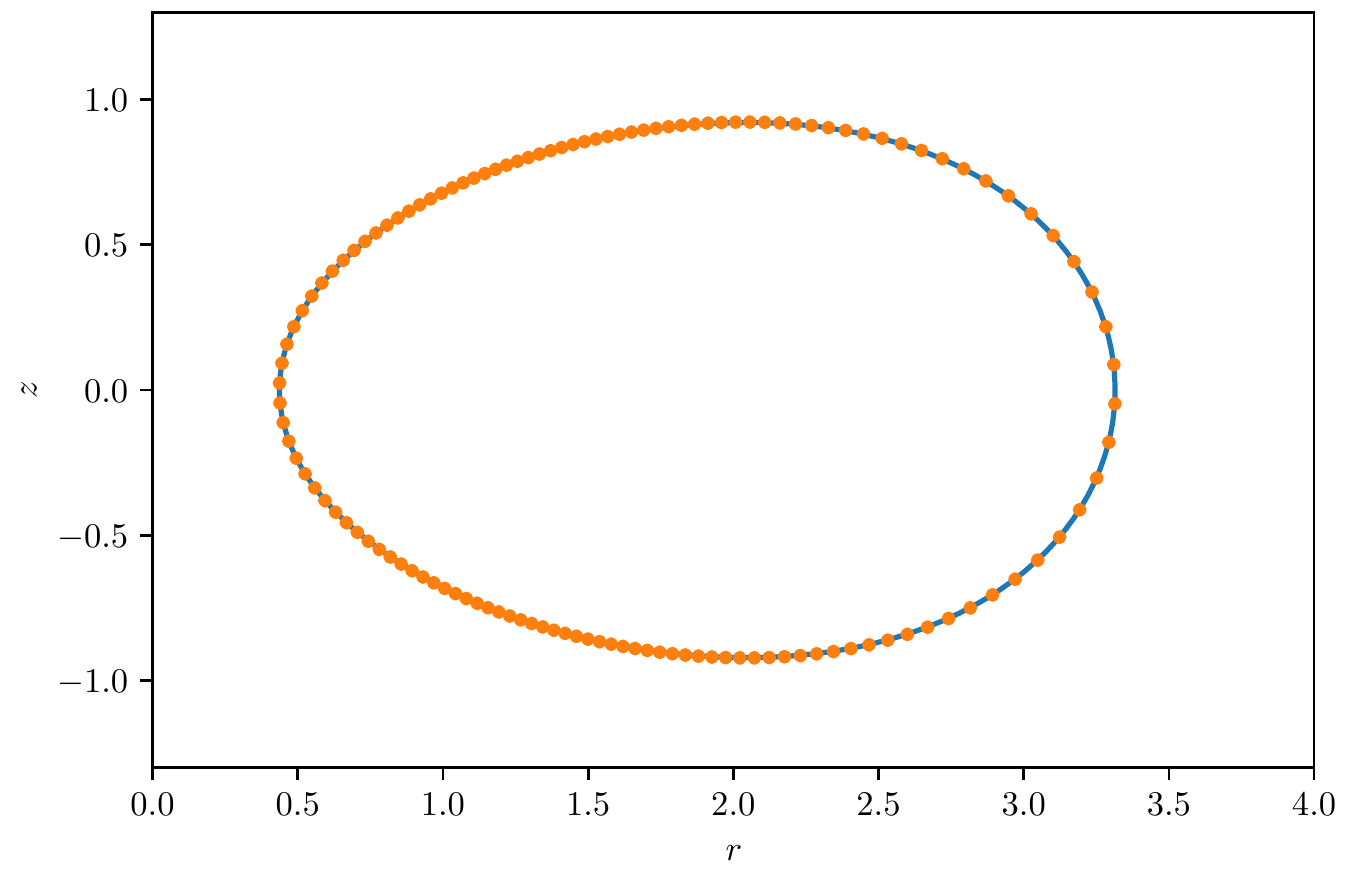}
  \caption{The Angenent torus cross-section (blue curve) and the discrete approximation to it (orange dots) computed in \cite{bk19}.}
  \label{fig:curvedots}
\end{figure}

We now proceed to compute a matrix approximation to the differential operator $L_0$, using equation~\eqref{eq:l0}. We can first rewrite this equation in terms of the metric $g^\sigma$.
\begin{proposition}
  Let $\Gamma$ be the cross-section of a self-shrinking torus, and let $u\colon\Gamma\to\R$. Then
  \begin{equation}\label{eq:l0length}
    \ddtwo s\Bigr\rvert_{s=0}\ell^\sigma(\Gamma_s)=\int_\Gamma u(-L_0)u\,d\ell^\sigma,
  \end{equation}
  where $\Gamma_s=\{x+su\n\mid x\in\Gamma\}$ is the normal variation corresponding to $u$, the expression $\ell^\sigma(\Gamma_s)$ denotes the length of the curve $\Gamma_s$ with respect to the conformally changed metric $g^\sigma=\sigma^2(dr^2+dz^2)$, and $d\ell^\sigma$ denotes integration with respect to $g^\sigma$ arc length.
\end{proposition}
In other words, with respect to the metric $g^\sigma$, the operator $-L_0$ is the Hessian of the length functional.

The task now is to make discrete approximations to all of the terms in equation~\eqref{eq:l0length}. We begin by approximating length the same way as in \cite{bk19}.

\subsection{Discrete length and its Hessian}
\begin{definition}
  Given two points $q_m$ and $q_{m+1}$, we approximate the distance between them with respect to the metric $g^\sigma$ by setting $\sigma_{\text{mid}}$ to be $\sigma$ evaluated at the midpoint $(q_m+q_{m+1})/2$, and then computing the discrete approximation to the distance to be
  \begin{equation}\label{eq:discretedistance}
    \dist^\sigma_d(q_m,q_{m+1}):=\sigma_{\text{mid}}\norm{q_{m+1}-q_m},
  \end{equation}
  where $\norm\cdot$ denotes the usual Euclidean norm.
\end{definition}

\begin{definition}
  Given a discrete curve $\Gamma_d=\{q_0,q_1,\dotsc,q_M=q_0\}$ that approximates a curve $\Gamma$, we approximate its length $\ell^\sigma(\Gamma)$ with the discrete length functional
  \begin{equation}\label{eq:discretelength}
    \ell^\sigma_d(\Gamma_d)=\sum_{m=0}^{M-1}\dist^\sigma_d(q_m,q_{m+1}).
  \end{equation}
\end{definition}

The computation of the Angenent torus cross-section $\Gamma_d=\{q_0,q_1,\dotsc,q_M=q_0\}$ in \cite{bk19} is set up so that $\Gamma_d$ is a critical point for the functional $\ell^\sigma_d$, analogously to the fact that the true cross-sectional curve $\Gamma$ is a critical point for the length functional $\ell^\sigma$. The corresponding critical value $\ell^\sigma_d(\Gamma_d)$ is an approximation of $\ell^\sigma(\Gamma)$, which is the entropy of the Angenent torus.

Thus, we can approximate the left-hand side of \eqref{eq:l0length} with
\begin{equation*}
  \ddtwo s\Bigr\rvert_{s=0}\ell^\sigma(\Gamma_s)\approx\ddtwo s\Bigr\rvert_{s=0}\ell_d^\sigma(\Gamma_{d;s}),
\end{equation*}
where $\Gamma_{d;s}$ is a variation of the discrete curve $\Gamma_d$. More precisely, $\Gamma_{d;s}=\{q_{0;s},q_{1;s},\dotsc,q_{M;s}=q_{0;s}\}$, where $q_{m;s}=q_m+sv_m$ for a sequence of vectors $v_0,v_1,\dotsc,v_M=v_0$ representing a discrete vector field $v_d$ on $\Gamma_d$.

We can think of $\ddtwo s\bigr\rvert_{s=0}\ell_d^\sigma(\Gamma_{d;s})$ as giving us the Hessian of $\ell^\sigma_d(q_0,q_1\dotsc,q_M=q_0)$.
\begin{definition}\label{def:hd}
  Let $\Gamma_d$ be a discrete curve. Viewing $\ell^\sigma_d$ as a function $(\R^2)^M=\R^{2M}\to\R$, let $H_d$ denote the Hessian of $\ell^\sigma_d$ at $\Gamma_d$, a $2M\times2M$ matrix.
\end{definition}
Then, by definition of the Hessian, if we view $v_d$ as a vector in $(\R^2)^M=\R^{2M}$, we have
\begin{equation}\label{eq:hessvd}
  \ddtwo s\Bigr\rvert_{s=0}\ell_d^\sigma(\Gamma_{d;s})=v_d^TH_dv_d.
\end{equation}

\subsection{The outward unit normal of a discrete curve}
Recall, however, that we restricted our attention to normal variations of $\Gamma$, because tangential variations do not change length. In other words, we considered only those variations $\Gamma_s=\{x+sv\mid x\in\Gamma\}$ where $v=u\n$ for some function $u\colon\Gamma\to\R$. To do the same for our discrete variations $v_d$ of our discrete curve $\Gamma_d$, we must appropriately define a normal vector field $\n_d=\{\n_0,\n_1,\dotsc,\n_M=\n_0\}$.

In this situation, there is a very natural way to do so by considering what happens to the discrete length $\ell^\sigma_d$ if we vary a single point $q_m$ of $\Gamma_d$ while leaving all of the other points fixed. Because $\Gamma_d$ is a critical point for $\ell^\sigma_d$, the first derivative is zero. Meanwhile, the second derivative is the $m$th $2\times2$ submatrix on the diagonal of the Hessian $H_d$, which we denote by $H_m$.

We expect that if we move $q_m$ in a tangential direction, either towards $q_{m-1}$ or $q_{m+1}$, then $\ell^\sigma_d$ will not change very much. On the other hand, if we move $q_m$ in a normal direction, then $\ell^\sigma_d$ will increase. Thus, we expect $H_m$ to have an eigenvalue close to zero, whose eigenvector approximates the direction tangent to the curve, and a positive eigenvalue, whose eigenvector approximates the normal direction tangent to the curve. We obtain the following definition of the normal vector field $\n_d$.

\begin{definition}\label{def:nm}
  Let $\Gamma_d$ be a discrete curve. Let $H_1,\dotsc,H_M$ be the $2\times2$ blocks along the diagonal of $H_d$. We define the \emph{outward unit normal vector field} $\n_d=\{\n_0,\n_1,\dotsc,\n_M=\n_0\}$ by letting $\n_m$ be the unit eigenvector corresponding to the larger eigenvalue of $H_m$, with the sign chosen so that $\n_m$ points outwards.
\end{definition}

A discrete function $u_d\colon\Gamma_d\to\R$ is just a discrete set of values $u_0,u_1,\dotsc,u_M=u_0$, with $u_m:=u_d(q_m)$. Now that we have our outward unit normal vector field, we can define $v_m=u_m\n_m$. Viewing $u_d$ as a vector in $\R^M$ and $v_d$ as a vector in $\R^{2M}$, this formula defines a linear transformation $\mbf N\colon\R^M\to\R^{2M}$.
\begin{definition}
  The matrix $\mbf N$ is an $2M\times M$ block diagonal matrix whose $2\times1$ diagonal blocks are the $\n_m$.
\end{definition}
Using $\mbf N$, we can rewrite equation~\eqref{eq:hessvd} as
\begin{equation}\label{eq:hessud}
  \ddtwo s\Bigr\rvert_{s=0}\ell^\sigma_d(\Gamma_{d;s})=u_d^T\left(\mbf N^TH_d\mbf N\right)u_d,
\end{equation}
where the variation $\Gamma_{d;s}$ is defined by $q_{m;s}=q_m+su_m\n_m$.

\subsection{Integrating over a discrete curve}
We now turn to the right-hand side of \eqref{eq:l0length}. To approximate it, we must first understand how to approximate integration with respect to $g^\sigma$ arc length, which is relatively straightforward because the points $q_m$ are equally spaced with respect to $g^\sigma$ arc length. We can view a discrete function $u_d\colon\Gamma_d\to\R$ as an approximation of a function $u\colon\Gamma\to\R$. Intuitively, because the $q_m$ are equally spaced, the $u_m$ should be weighted equally, so the average value of the $u_m$ should approximate the average value of $u$. In other words, we expect
\begin{equation*}
  \frac1{\ell^\sigma(\Gamma)}\int_\Gamma u\,d\ell^\sigma\approx\frac1M\sum_{m=0}^{M-1}u_m.
\end{equation*}
To understand this approximation more precisely, we can let $q(t)$ denote the parametrization of $\Gamma$ with respect to $g^\sigma$ arc length. Because the $q_m$ are equally spaced, setting $\Delta t=\ell^\sigma(\Gamma)/M$, we have that $q_m$ is an approximation for $q(m\Delta t)$, and hence $u_m$ is an approximation for $u(q(m\Delta t))$. Thus,
\begin{multline*}
  \int_\Gamma u\,d\ell^\sigma=\int_0^{\ell^\sigma(\Gamma)}u(q(t))\,dt=\sum_{m=0}^{M-1}\int_{\left(m-\frac12\right)\Delta t}^{\left(m+\frac12\right)\Delta t}u(q(t))\,dt\\
  \approx\sum_{m=0}^{M-1}u(q(m\Delta t))\Delta t\approx\sum_{m=0}^{M-1}u_m\frac{\ell^\sigma(\Gamma)}M\approx\frac{\ell^\sigma_d(\Gamma_d)}M\sum_{m=0}^{M-1}u_m.
\end{multline*}

\subsection{The discrete stability operator}
Applying these approximations to \eqref{eq:l0length}, we can approximate $L_0$ with a discrete operator $L_{0;d}$.
\begin{definition}
  Viewing a discrete function $u_d=\{u_0,u_1,\dotsc,u_M=u_0\}$ as an element of $\R^M$, let $L_{0;d}$ be the $M\times M$ symmetric matrix satisfying
  \begin{equation*}
    \ddtwo s\Bigr\rvert_{s=0}\ell^\sigma_d(\Gamma_{d;s})=\frac{\ell^\sigma_d(\Gamma_d)}M\sum_{m=0}^{M-1}u_m(-L_{0;d}u_d)_m
  \end{equation*}
  for all $u_m$, where $\Gamma_{d;s}$ is defined by $q_{m;s}=q_m+su_m\n_m$.
\end{definition}

\begin{proposition}
  \begin{equation*}
    -L_{0;d}=\frac M{\ell^\sigma_d(\Gamma_d)}\mbf N^TH_d\mbf N.
  \end{equation*}
\end{proposition}
\begin{proof}
  Using equation~\eqref{eq:hessud} for the left-hand side and rewriting the right-hand side, we have
  \begin{equation*}
    u_d^T\left(\mbf N^TH_d\mbf N\right)u_d=u_d^T\left(-\frac{\ell^\sigma_d(\Gamma_d)}ML_{0;d}\right)u_d.
  \end{equation*}
  The result follows using the fact that the Hessian $H_d$ is symmetric.
\end{proof}

From here, it is easy to approximate the $k$th Fourier component of the stability operator. Recall that $-L_k=-L_0+\frac{k^2}{r^2}$. Letting $r_m$ denote the $r$-coordinate of $q_m$, we can approximate the operator $\frac{k^2}{r^2}$ with the diagonal matrix whose entries are $\frac{k^2}{r_m^2}$.
\begin{definition}
  Let $L_{k;d}$ be the $M\times M$ matrix defined by
  \begin{equation*}
    -L_{k;d}=-L_{0;d}+k^2R_d^{-2},
  \end{equation*}
  where $R_d$ is the diagonal $M\times M$ matrix whose entries are the $r_m$.
\end{definition}

By computing the eigenvalues of these matrix approximations $L_{k;d}$ of $L_k$, we can estimate the eigenvalues of $L_k$. Counting the number of negative eigenvalues will give us the index of the Angenent torus.

\subsection{Implementation}
In practice, we do not compute the Hessian $H_d$ all at once. Instead, our first step is to use \texttt{sympy} to symbolically compute the Hessian of $\dist^\sigma_d\colon\R^2\times\R^2\to\R$, giving us a $4\times4$ matrix of expressions. Then, for each $m$, we evaluate these expressions at $(q_m,q_{m+1})$, giving us a $4\times4$ matrix that we denote by $H_{m,m+1}$. While we could assemble the $H_{m,m+1}$ into the full matrix $H_d$ by placing the $4\times4$ blocks $H_{m,m+1}$ in appropriate locations along the diagonal and adding them together, we instead first compute the normal vectors and use them to reduce the dimension of the problem.

Recall that we compute the normal vector $\n_m$ by seeing how $\ell^\sigma_d(\Gamma_d)$ changes as we vary the single point $q_m$. Since $\dist^\sigma_d(q_{m-1},q_m)+\dist^\sigma_d(q_m,q_{m+1})$ are the only two terms in $\ell^\sigma_d(\Gamma_d)$ where $q_m$ appears, we only need to know $H_{m-1,m}$ and $H_{m,m+1}$. Adding the bottom right $2\times2$ block of $H_{m-1,m}$ to the top left $2\times2$ block of $H_{m,m+1}$, we obtain the $2\times2$ matrix $H_m$ from Definition~\ref{def:nm}, so $\n_m$ is the unit eigenvector corresponding to the larger eigenvalue of this $2\times2$ matrix.

Next, we would like to reduce the dimension of $H_{m,m+1}$ by considering normal variations only. To do so, we assemble $\n_m$ and $\n_{m+1}$ into a $4\times2$ block diagonal matrix $\mbf N_{m,m+1}$, from which we obtain the $2\times2$ matrix
\begin{equation*}
  -L_{0;m,m+1}:=\frac M{\ell^\sigma_d(\Gamma_d)}\mbf N_{m,m+1}^TH_{m,m+1}\mbf N_{m,m+1}.
\end{equation*}
Essentially, $-L_{0;m,m+1}$ represents how the distance $\dist^\sigma_d(q_m,q_{m+1})$ changes if we vary $q_m$ and $q_{m+1}$ in the normal directions, weighted by $\frac M{\ell^\sigma_d(\Gamma_d)}\approx\dist^\sigma_d(q_m,q_{m+1})^{-1}$.

Finally, we assemble the $2\times2$ matrices $-L_{0;m,m+1}$ into the $M\times M$ matrix $-L_{0;d}$ by placing $-L_{0;m,m+1}$ into the $2\times2$ block formed by the $m$th and $(m+1)$st rows and columns, and then summing over $m$. From here, it is an easy matter to obtain $-L_{k;d}:=-L_{0;d}+k^2R_d^{-2}$ by using the $r$-coordinates of the $q_m$ to assemble the diagonal matrix $R_d^{-2}$. Once we have the matrices $-L_{k;d}$, we can compute their eigenvalues and eigenvectors with \texttt{numpy}.

See the supplementary materials for a Jupyter notebook implementation. 

\section{Results}\label{sec:results}
\begin{figure}
  \centerline{\includegraphics[trim = 0 6.8in 0 0, clip]{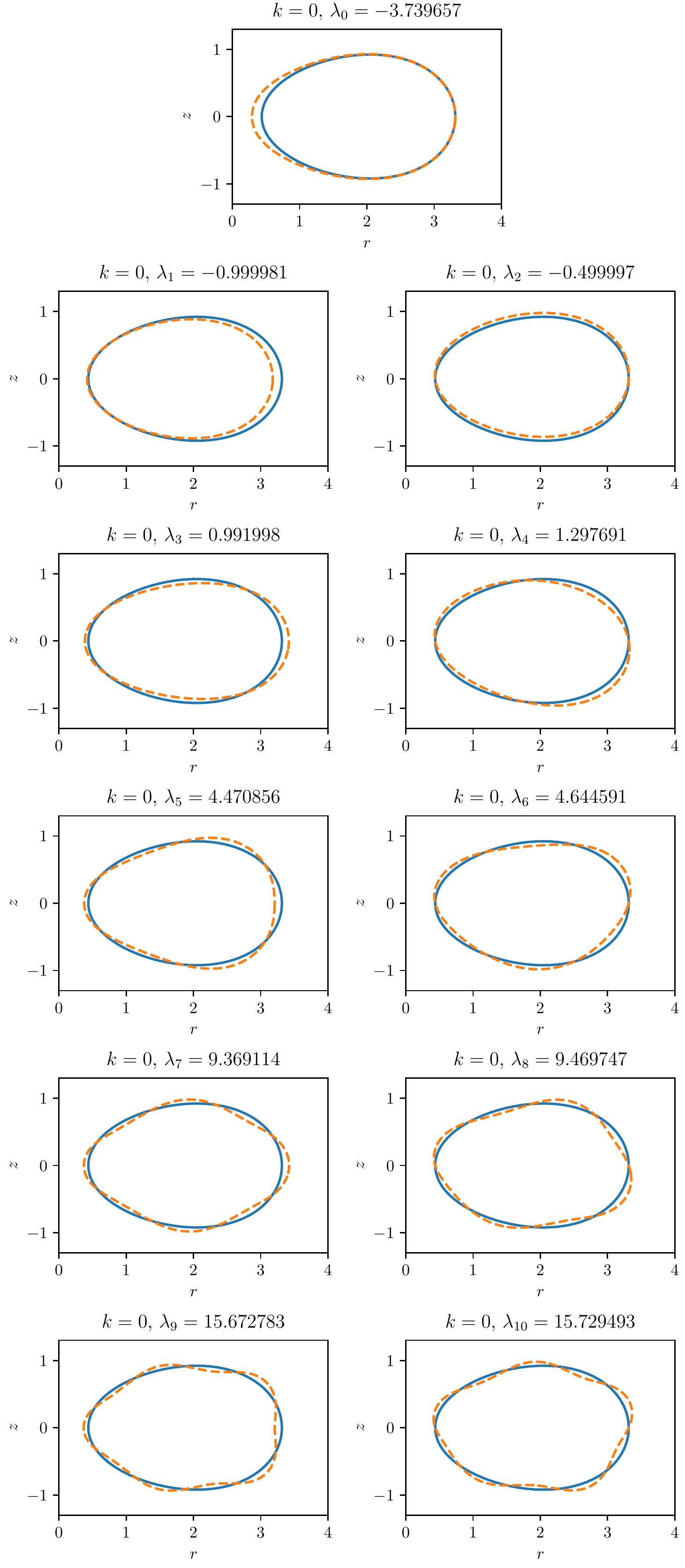}}
  \caption{The first few eigenvalues and eigenfunctions of $L_0$. The eigenfunctions are pictured as variations (orange dashed) of the Angenent torus cross-section (blue solid). The variation corresponding to $\lambda_1=-1$ is dilation, and the variation corresponding to $\lambda_2=-\frac12$ is vertical translation.}
  \label{fig:k0}
\end{figure}

\begin{figure}
  \centerline{\includegraphics{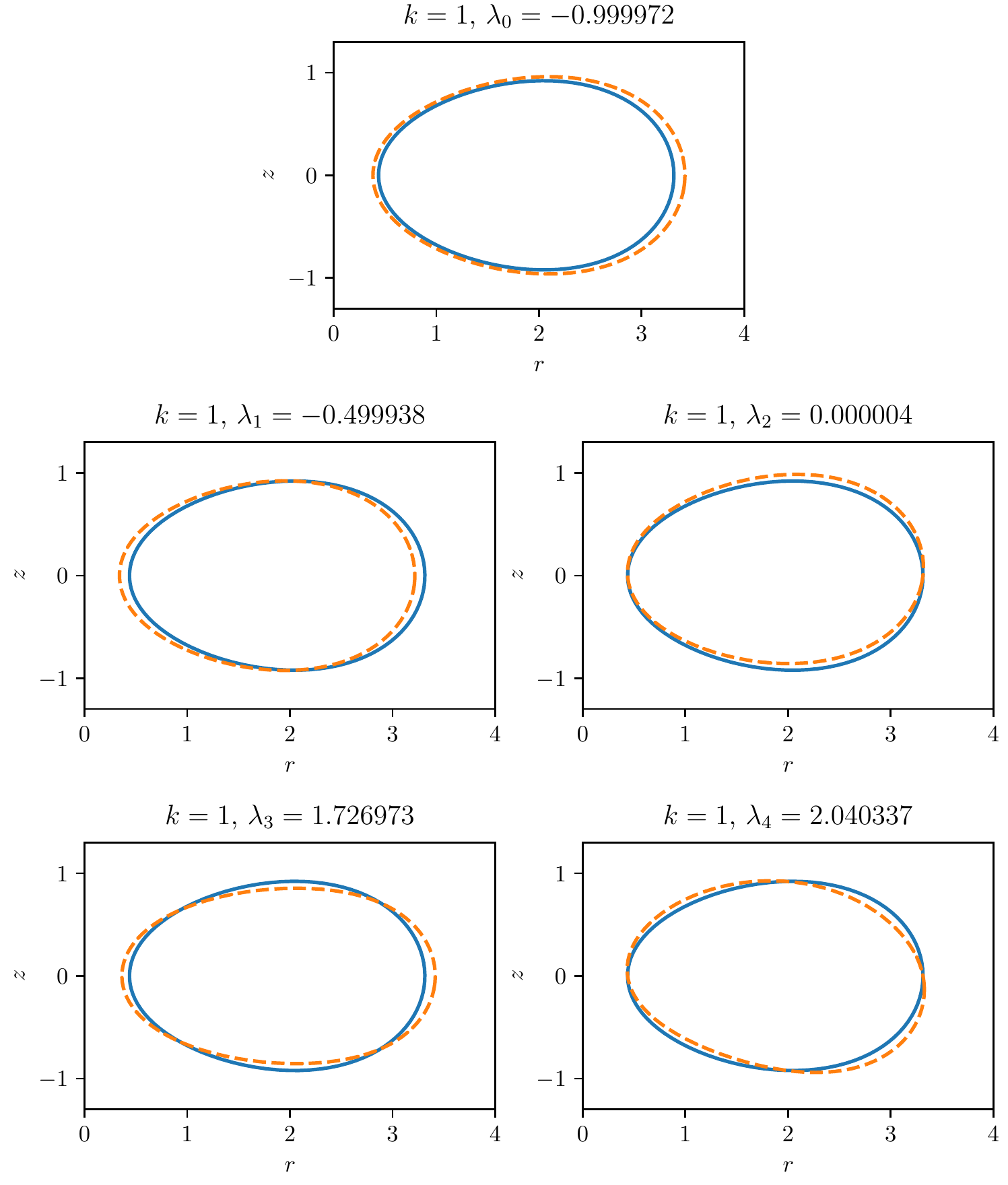}}
  \caption{The first few eigenvalues and eigenfunctions of $L_1$. The variation corresponding to $\lambda_0=-1$ is $\sigma^{-1}$ \cite[Section 6]{bk20a}. The variation corresponding to $\lambda_1=-\frac12$ is horizontal translation, and the variation corresponding to $\lambda_2=0$ is rotation about the origin.}
  \label{fig:k1}
\end{figure}
\begin{figure}
  \centerline{\includegraphics{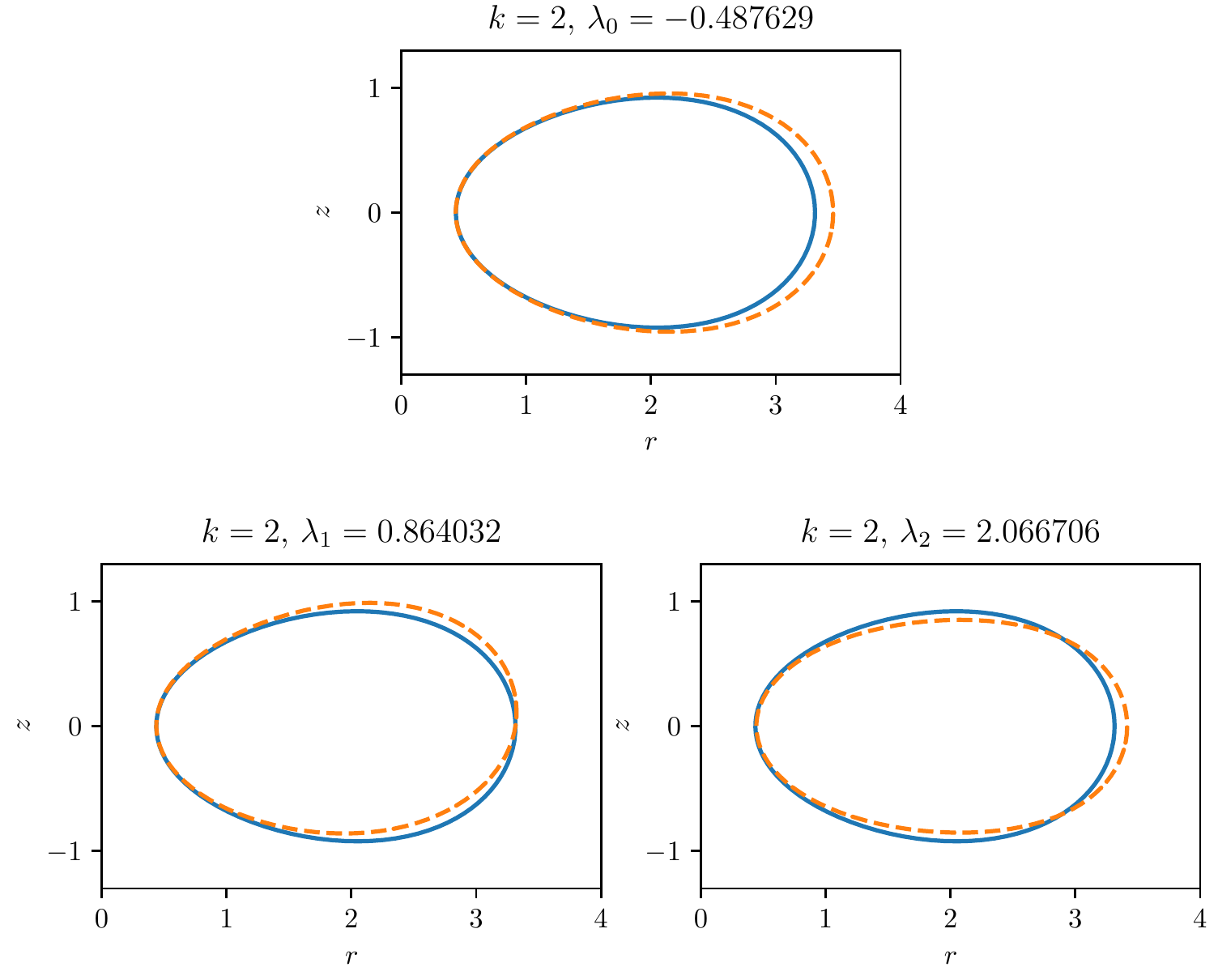}}
  \caption{The first few eigenvalues and eigenfunctions of $L_2$.}
  \label{fig:k2}
\end{figure}
We estimate the first few eigenvalues and eigenfunctions of $L_k$ by computing the eigenvalues and eigenvectors of $L_{k;d}$ with $M=2048$. We present these variations of the Angenent torus cross-section in Figures~\ref{fig:k0}--\ref{fig:k2}. Recall that for $k>0$, each eigenfunction $u$ of $L_k$ corresponds to two eigenfunctions $u\cos k\theta$ and $u\sin k\theta$ of the stability operator $L_\Sigma$. The variations with negative eigenvalues, excluding translation and dilation, contribute to the index. We present three-dimensional plots of the variations with negative eigenvalues in Figure~\ref{fig:3d}. There are $9$ such variations. One of the variations is dilation, and three are translations, so the entropy index of the Angenent torus is $5$.

\begin{figure}
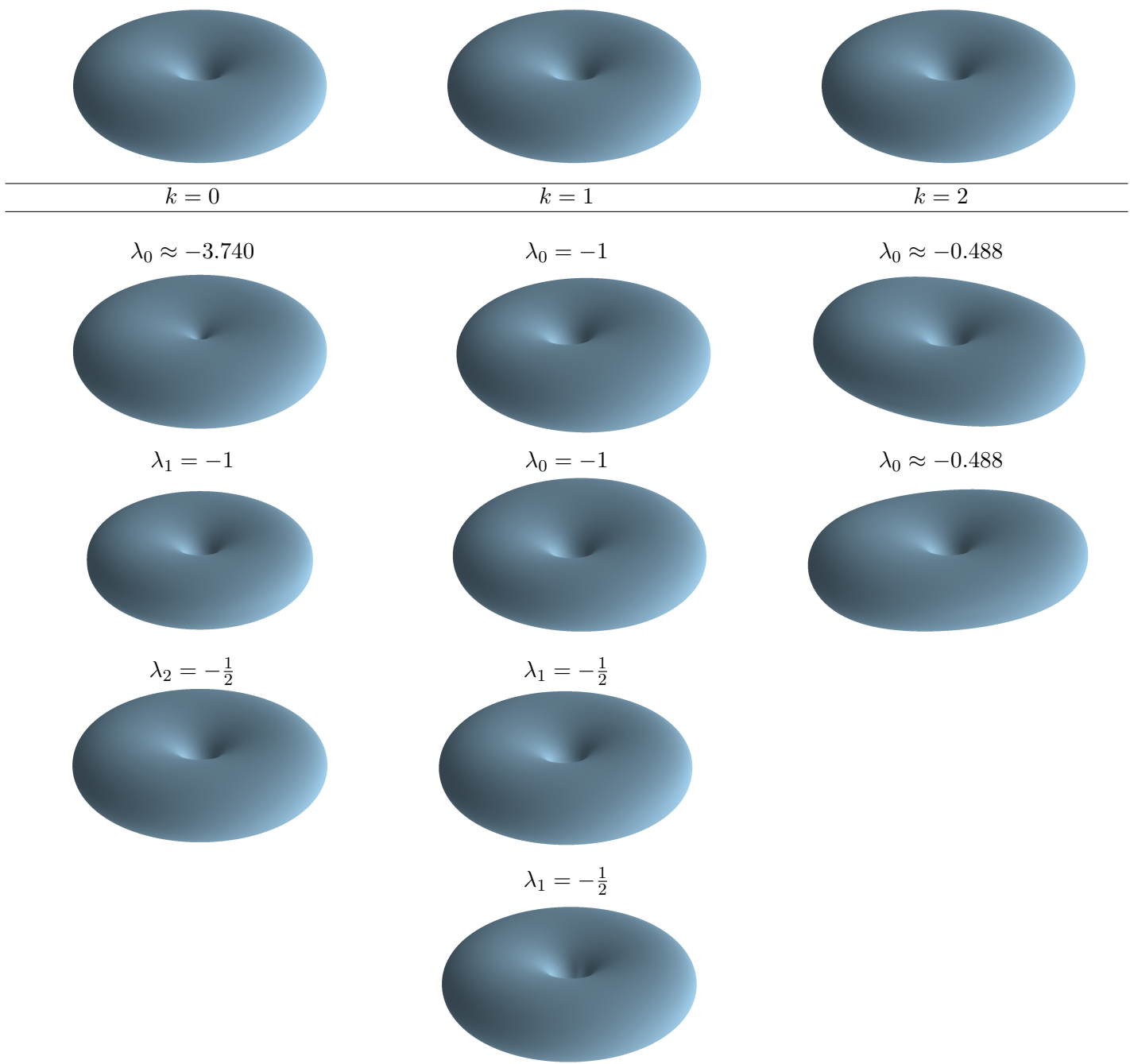

  \newcommand\plot[1]{\includegraphics[trim = 30mm 54mm 30mm 52mm, clip, scale=.6]{figures/#1\res res}}
  \newcommand\plotv[1]{\plot{variation3D#1}}
  \centerline{
    \begin{tabular}{ccc}
      \plot{torus3D}&\plot{torus3D}&\plot{torus3D}\\\hline
      $k=0$&$k=1$&$k=2$\\\hline\\
      $\lambda_0\approx-3.740$&$\lambda_0=-1$&$\lambda_0\approx-0.488$\\
      \plotv{00}&\plotv{10c}&\plotv{20c}\\
      $\lambda_1=-1$&$\lambda_0=-1$&$\lambda_0\approx-0.488$\\
      \plotv{01}&\plotv{10s}&\plotv{20s}\\
      $\lambda_2=-\frac12$&$\lambda_1=-\frac12$\\
      \plotv{02}&\plotv{11c}&\\
                    &$\lambda_1=-\frac12$\\
                &\plotv{11s}
  \end{tabular}
  }
  \caption{The Angenent torus (top row) and its variations with negative eigenvalues. In the first column, we have dilation with eigenvalue $-1$ and vertical translation with eigenvalue $-\frac12$. In the second column, we have the pair of variations with eigenvalue $-1$ discussed in \cite[Section 6]{bk20a}, and the two horizontal translations with eigenvalue $-\frac12$.}
  \label{fig:3d}
\end{figure}

\section{Error analysis}\label{sec:error}
To have confidence in our index result, we must approximate the eigenvalues of $L_k$ to sufficient accuracy to know that values we compute to be negative are indeed negative, and that the lowest computed positive eigenvalue of each $L_k$ is indeed positive. Eigenvalues close to zero could pose a problem, but the only such eigenvalue is $\lambda_2$ for $k=1$. Fortunately, for this eigenvalue, we know that its true value is exactly $\lambda_2=0$, corresponding to the variations that rotate the Angenent torus about the $x$ or $y$ axes.

To estimate the error in our values, we perform our computation with different numbers of points $M$. Specifically, we use $M\in\{128, 256, 512, 1024, 2048\}$. When we know the true value, we can directly see how fast our estimate converges by plotting the logarithm of the error against the logarithm of $M$. When the true value is unknown, we estimate it with the value that results in the best linear fit on a log-log plot. We present our results in Figure~\ref{fig:errorplots} and Table~\ref{tab:errorplots}

\begin{figure}
  \centerline{\includegraphics[scale=.9]{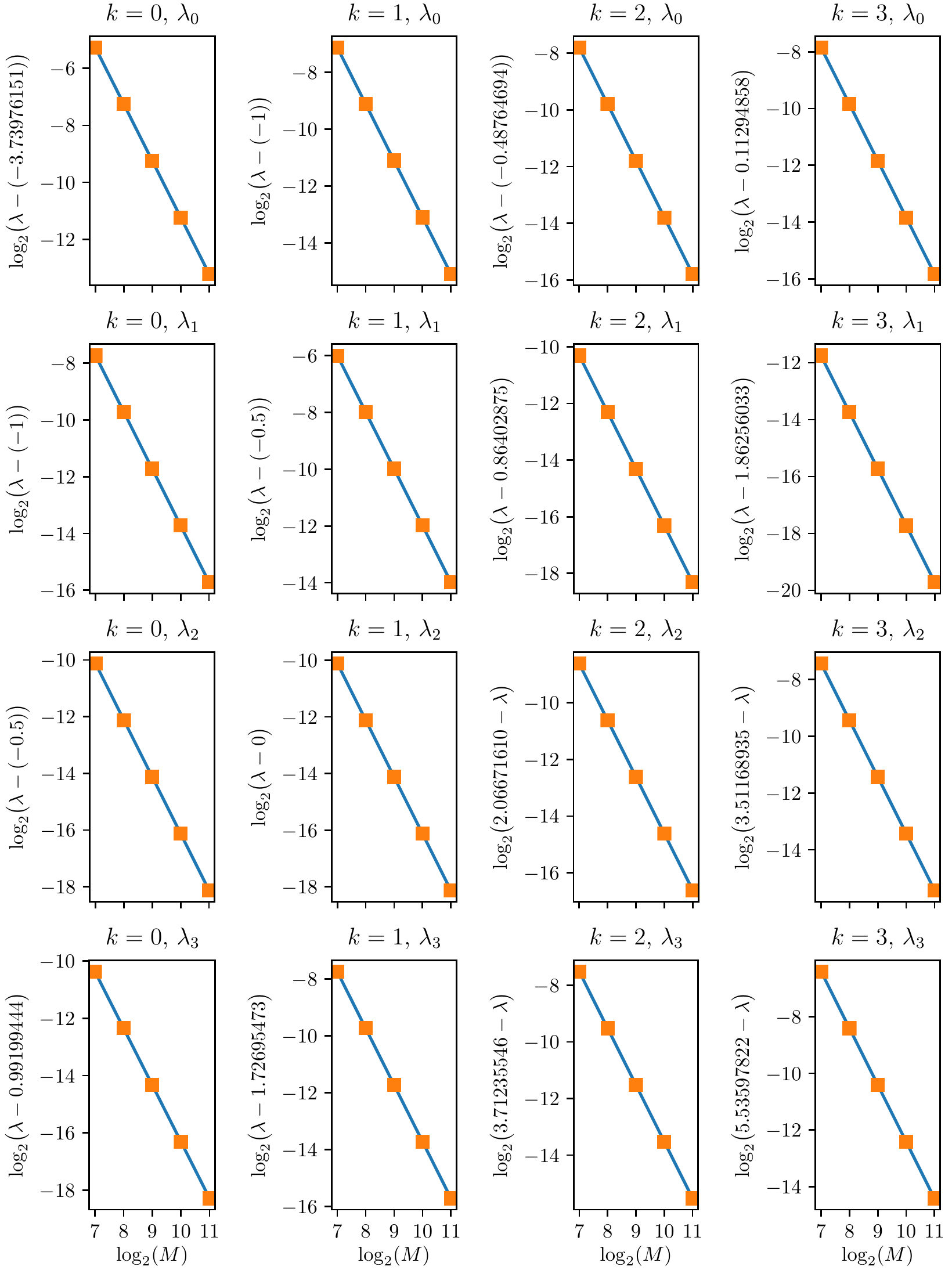}}
  \caption{Log-log plots showing the rate at which the error in our eigenvalue estimates decreases as the number of points $M$ increases.}
  \label{fig:errorplots}
\end{figure}

\begin{table}
  \centering
  \begin{tabular}{c|c|ccc}
    &&\makecell{value computed\\with $M=2048$}&\makecell{true value (known or\\ estimated from fit)}&error\\\hline
    \multirow4*{$k=0$}
    &$ \lambda_{0} $& $-3.73965698$&$ -3.73976151 $&$ \phantom-1.0\times10^{-4}$\\
    &$ \lambda_{1} $&$ -0.99998145 $&$ -1 $&$ \phantom-1.9\times10^{-5}$\\
    &$ \lambda_{2} $&$ -0.49999650 $&$ -\frac12 $&$ \phantom-3.5\times10^{-6}$\\
    &$ \lambda_{3} $&$ \phantom-0.99199758 $&$ \phantom-0.99199444 $&$ \phantom-3.1\times10^{-6}$\\\hline
    \multirow4*{$k=1$}
    &$ \lambda_{0} $&$ -0.99997152 $&$ -1 $&$ \phantom-2.8\times10^{-5}$\\
    &$ \lambda_{1} $&$ -0.49993807 $&$ -\frac12 $&$ \phantom-6.2\times10^{-5}$\\
    &$ \lambda_{2} $&$ \phantom-0.00000351 $&$ \phantom-0 $&$ \phantom-3.5\times10^{-6}$\\
    &$ \lambda_{3} $&$ \phantom-1.72697331 $&$ \phantom-1.72695473 $&$ \phantom-1.9\times10^{-5}$\\\hline
    \multirow4*{$k=2$}
    &$ \lambda_{0} $&$ -0.48762926 $&$ -0.48764694 $&$ \phantom-1.8\times10^{-5}$\\
    &$ \lambda_{1} $&$ \phantom-0.86403182 $&$ \phantom-0.86402875 $&$ \phantom-3.1\times10^{-6}$\\
    &$ \lambda_{2} $&$ \phantom-2.06670611 $&$ \phantom-2.06671610 $&$ -1.0\times10^{-5}$\\
    &$ \lambda_{3} $&$ \phantom-3.71233427 $&$ \phantom-3.71235546 $&$ -2.1\times10^{-5}$\\\hline
    \multirow4*{$k=3$}
    &$ \lambda_{0} $&$ \phantom-0.11296571 $&$ \phantom-0.11294858 $&$ \phantom-1.7\times10^{-5}$\\
    &$ \lambda_{1} $&$ \phantom-1.86256149 $&$ \phantom-1.86256033 $&$ \phantom-1.2\times10^{-6}$\\
    &$ \lambda_{2} $&$ \phantom-3.51166663 $&$ \phantom-3.51168935 $&$ -2.3\times10^{-5}$\\
    &$ \lambda_{3} $&$ \phantom-5.53593246 $&$ \phantom-5.53597822 $&$ -4.6\times10^{-5}$\\
  \end{tabular}
  \vspace\baselineskip
  \caption{For the first four eigenvalues of each of the first four $L_k$, we give the error between the eigenvalue computed with $M=2048$ and either the true value if known or our estimate of the true value based on the fits in Figure~\ref{fig:errorplots}.}
  \label{tab:errorplots}
\end{table}

The slopes of the best fit lines in Figure~\ref{fig:errorplots} are all between $-1.978$ and $-2.002$, suggesting a quadratic rate of convergence. This rate of convergence matches the expected quadratic rate of convergence for the entropy of the Angenent torus that we found in \cite{bk19}. We observe that the computed eigenvalues sometimes overestimate and sometimes underestimate the true value, and that the magnitude of the error varies. In all cases, however, there is by far more than enough accuracy to determine the sign of the eigenvalues. Because the $j$th eigenvalue of $L_k$ increases with both $j$ and $k$, we know that Table~\ref{tab:errorplots} lists all of the negative eigenvalues; there can be no others.

\section{Future work}\label{sec:future}
There are several directions for future work, of varying levels of complexity.
\subsection{Other rotationally symmetric self-shrinkers}
The methods in \cite{bk19} and in this paper can be immediately applied to any other rotationally symmetric surface, of which there are infinitely many examples \cite{dk17}. Mramor's work \cite{m20} suggests that the entropy of these examples should grow to infinity. Meanwhile, our work \cite{bk20a} shows that the index should grow at least linearly with the entropy, but our upper bound on the index allows for faster growth. Numerically computing the entropies and indices of these examples could give some insight about whether these index bounds are optimal in terms of asymptotic growth rate, or if they could be improved.

\subsection{Self-shrinkers without symmetry}
Rotational symmetry allows us to reduce the dimension of the problem: Rather than computing variations of a critical surface in $\R^3$, we can compute variations of a critical curve in the half-plane, which allows us to work with ordinary differential equations rather than partial differential equations. However, by using numerical methods for working with partial differential equations, we could analyze the general problem without rotational symmetry in much the same way. Namely, we could discretize the problem by triangulating the surface. We could then approximate the $F$-functional by summing the weighted areas of the triangles, giving us a functional on a finite-dimensional space. Finally, we could compute the critical points of this functional and compute the Hessian.

\subsection{Error analysis}
We have strong numerical evidence that the values we obtained in \cite{bk19} and in this paper are accurate. When true values are known, our methods find them. Even when true values are unknown, we observe a quadratic rate of convergence as we increase the number of points. Additionally, the value of the entropy in \cite{bk19} has since been reproduced using different numerical methods \cite{bdn19}. Nonetheless, numerical evidence does not constitute a proof, so it would be good to prove error bounds on our estimates.

The starting point would be to consider our estimate $\dist^\sigma_d(q_m,q_{m+1})$ for the distance between two points $q_m$ and $q_{m+1}$ in the half-plane $Q=\{(r,z)\mid r\ge0\}$ with respect to the metric $g^\sigma$. We would like to bound the difference between this estimate and the true distance $\dist^\sigma(q_m,q_{m+1})$. One can do so by computing the Taylor polynomials of the distance squared at the diagonal of $Q\times Q$. From there, we could bound the error for the length functional, critical curve, Hessian, and so forth. One caveat is that variations of the discrete curve do a poor job of capturing variations of the true curve that are highly oscillatory, as we illustrate in Appendix~\ref{sec:asymptotics}. This issue is resolved by the fact that, as we show in \cite{bk20a}, highly oscillatory variations must increase length and therefore cannot contribute to the index.

\bibliographystyle{plain}
\bibliography{meancurvature}

\appendix
\section{Eigenvalue asymptotics}\label{sec:asymptotics}
For computing the index, we were concerned with eigenvalues $\lambda_j$ of $L_k$ for small values of $j$ and $k$. In this appendix, we go in the other direction and take a cursory look at the asymptotic behavior as either $j$ or $k$ becomes large.

In Figure~\ref{fig:k0cont}, we extend Figure~\ref{fig:k0}, showing the next few eigenvalues and eigenfunctions of $L_0$. We see that the eigenfunctions resemble the vibrational modes of a string. Meanwhile, in Figure~\ref{fig:j0}, we show the eigenfunction corresponding to the least eigenvalue of $L_k$ for the first several values of $k$. We see that the eigenfunctions become concentrated at the outermost point of the torus cross-section.
\begin{figure}
  \centerline{\includegraphics[trim = 0 0 0 6.7in, clip]{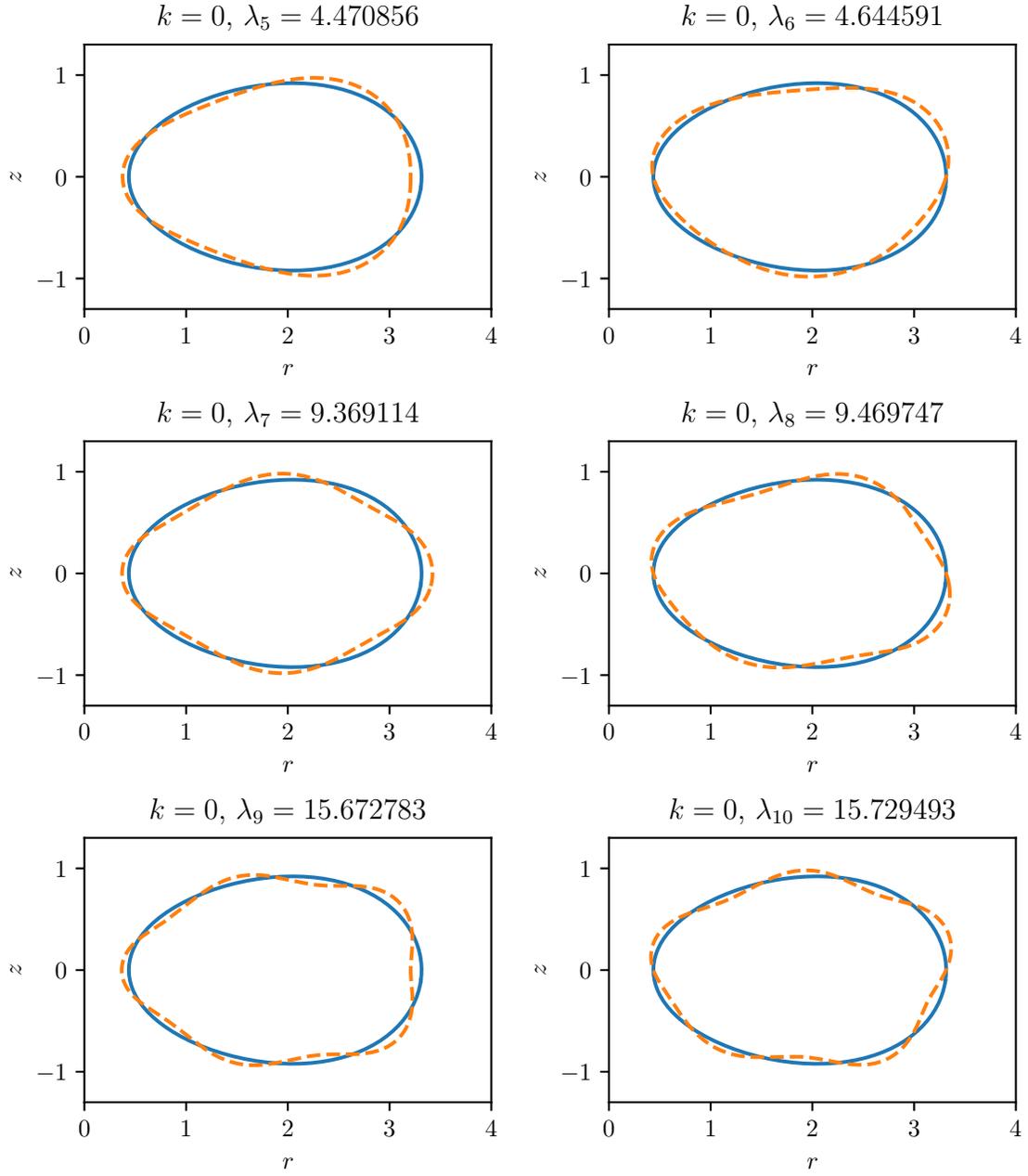}}
  \caption{A continuation of Figure~\ref{fig:k0}, showing the eigenvalues and eigenfunctions of $L_0$.}
    \vspace{1in}
  \label{fig:k0cont}
\end{figure}
\begin{figure}
  \centerline{\includegraphics[scale=.7]{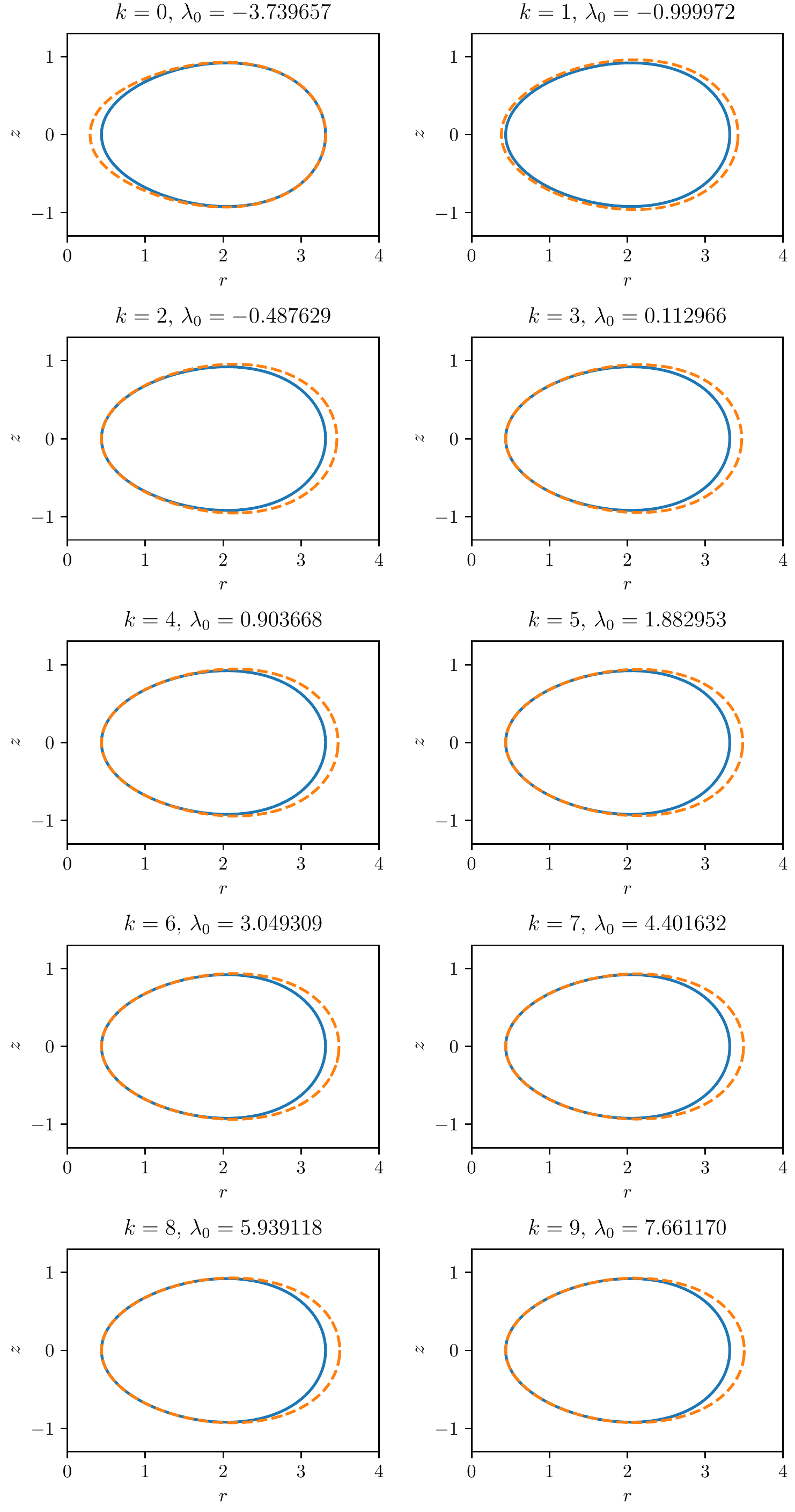}}
  \caption{The lowest eigenvalue and corresponding eigenfunction of $L_k$ for the first several values of $k$.}
  \label{fig:j0}
\end{figure}

To understand this behavior in greater detail, we apply a Liouville transformation \cite{am04, l37} to the formula for the operator $L_k$ given in Proposition~\ref{prop:lk}. As before, we let the variable $t$ parametrize the cross-sectional curve $\Gamma$ with respect to $g^\sigma$ arc length. We also let the variable $s$ parametrize the curve with respect to Euclidean arc length. We use the notation $\dot\psi=\dd[\psi]t$ and $\psi'=\dd[\psi]s$. Note that $\dd[t]s=\sigma$, so $\psi'=\dot\psi\sigma$. With this notation, Proposition~\ref{prop:lk} tells us that $u$ is an eigenfunction of $L_k$ if
\begin{equation}\label{eq:lku}
  L_ku+\lambda u=\sigma\ddtwo t(\sigma u)+\left(1+\frac{1-k^2}{r^2}+\lambda\right)u=0.
\end{equation}
Applying the Liouville transformation $u(t)=\psi(s)\sigma^{-1/2}$, we can compute that $u$ is a solution to equation~\eqref{eq:lku} if and only if $\psi$ is a solution to
\begin{equation}\label{eq:lkpsi}
  \psi''+\left(-\left(\sigma^{-1/2}\right)''\sigma^{1/2}+1+\frac{1-k^2}{r^2}+\lambda\right)\psi=0.
\end{equation}
We can identify equation~\eqref{eq:lkpsi} as a Schr\"odinger equation
\begin{equation*}
  H\psi=\lambda\psi
\end{equation*}
with Hamiltonian
\begin{equation*}
  H=-\ddtwo s + V,
\end{equation*}
where $V\colon\Gamma\to\R$ is the potential
\begin{equation*}
  V:=\left(\sigma^{-1/2}\right)''\sigma^{1/2}-1+\frac{k^2-1}{r^2}.
\end{equation*}

\subsection{Asymptotics for large $j$}
For high-frequency modes when $\lambda$ is large, the kinetic energy term $-\ddtwo s$ is more significant than the bounded potential energy term $V$. Based on this intuition and on \cite{am04, h75}, we can approximate $V$ with its average value
\begin{equation*}
  V_{\text{avg}}:=\frac1{\ell(\Gamma)}\int_\Gamma V\,ds,
\end{equation*}
where $\ell(\Gamma)$ is the length of $\Gamma$ with respect to the Euclidean metric. Thus, we have
\begin{equation*}
  H\approx-\ddtwo s + V_{\text{avg}},
\end{equation*}
and so the eigenvalues of $H$ are approximately
\begin{equation*}
  \lambda_{2j-1}\approx\lambda_{2j}\approx\left(\frac{2\pi}{\ell(\Gamma)}\right)^2j^2+V_{\text{avg}}.
\end{equation*}
Based on \cite{am04,h75}, we expect this approximation to be accurate to $O\left(\frac1{j^3}\right)$.

Using our discrete curve $\Gamma_d$, we can approximate $V_{\text{avg}}$ in a straightforward way. Since the points of $\Gamma_d$ are equally spaced with respect to $t$, we can approximate derivatives and integrals with respect to $t$ using differences and sums. We can easily compute $\sigma$ along the discrete curve, and then using $\dd s =\sigma\dd t$ and $ds=\sigma^{-1}\,dt$, we can approximate derivatives and integrals with respect to $s$ as well.

With this approximation for $V_{\text{avg}}$, we can look at how well our numerically computed eigenvalues $\lambda_{2j-1}$ and $\lambda_j$ of $L_k$ compare with the estimate $\left(\frac{2\pi}{\ell(\Gamma)}\right)^2j^2+V_{\text{avg}}$. We illustrate our findings in the case $k=0$ in Figure~\ref{fig:highj}.

\begin{figure}
  \centering
  \includegraphics{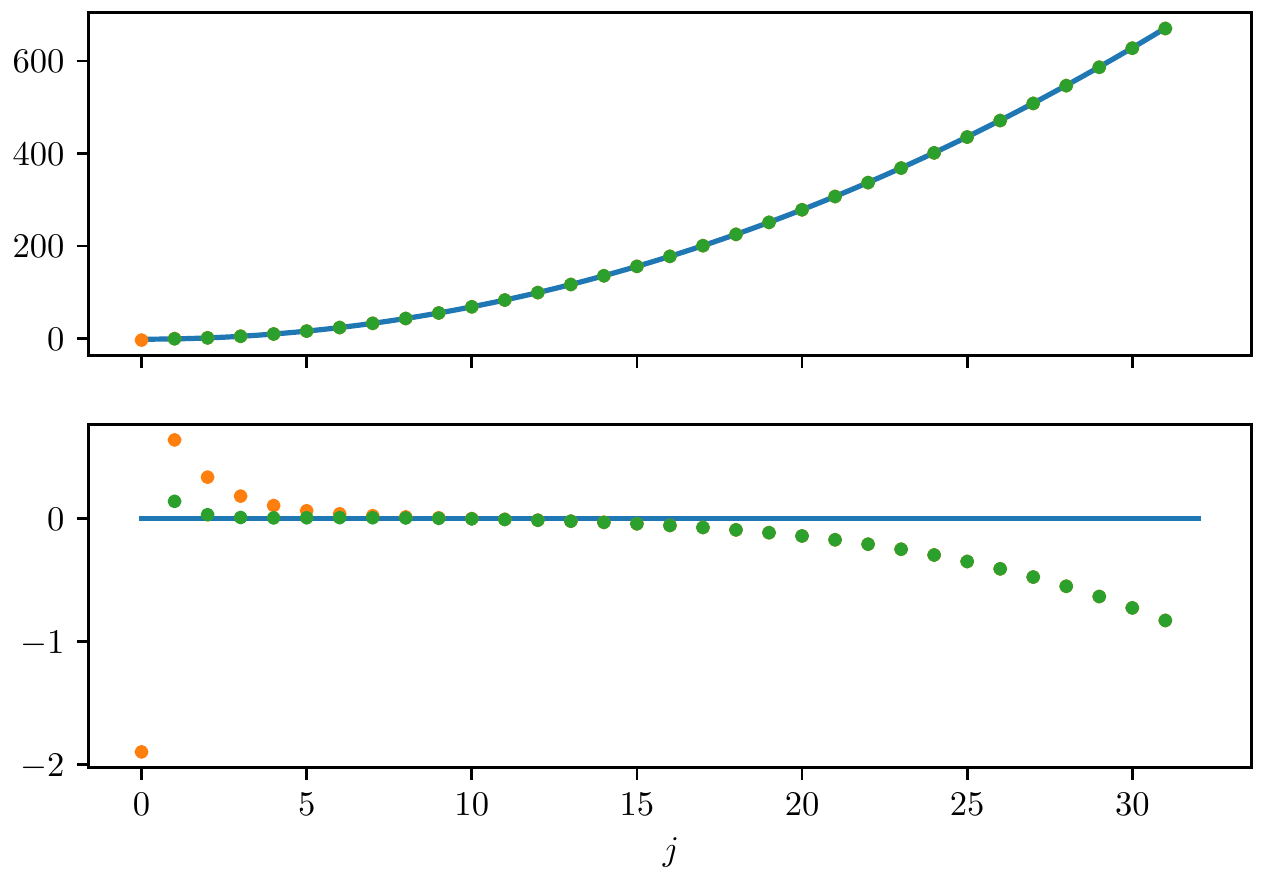}
  \caption{In the upper plot, we plot the eigenvalues $\lambda_{2j-1}$ (green dots) and $\lambda_{2j}$ (orange dots) of our approximation $L_{0;d}$ to the stability operator $L_0$, along with our asymptotic approximation $\left(\frac{2\pi}{\ell(\Gamma)}\right)^2j^2+V_{\text{avg}}$ (blue curve). In the lower plot, we plot the difference between the eigenvalues of $L_{0;d}$ and the asymptotic approximation.}
  \label{fig:highj}
\end{figure}

We observe that the relative error between our numerically computed eigenvalues and our asymptotic estimate is small. However, in absolute terms, we observe that our numerically computed eigenvalues eventually start falling below the asymptotic curve as $O(j^4)$, in contrast to the $O(\frac1{j^3})$ convergence suggested by \cite{am04, h75}. The explanation is that \cite{am04,h75} considers continuous systems, so the results apply to the eigenvalues of $L_0$. However, as our eigenfunctions oscillate faster, the discrete nature of our system becomes more apparent, and so, as we look at larger and larger eigenvalues, the eigenvalues of $L_{0;d}$ start to drift away from the corresponding eigenvalues of $L_0$.

The phenomenon that large eigenvalues of a discrete system deviate from the eigenvalues of the corresponding continuous system appears to be well-known; see for example \cite[Section 2.4]{c68} and \cite{s12}. Either reference has an example where the eigenvalues grow as $\left(\frac{2M}{\ell(\Gamma)}\sin\left(\frac\pi Mj\right)\right)^2$. This expression approaches $\left(\frac{2\pi}{\ell(\Gamma)}\right)^2j^2$ as $M\to\infty$, but, for fixed $M$, it deviates from $\left(\frac{2\pi}{\ell(\Gamma)}\right)^2j^2$ as $O(j^4)$, the same rate that we observe in our work. Note, however, that the coefficient in front of $j^4$ that we empirically observe is close but not equal to the coefficient that we would expect from the formula $\left(\frac{2M}{\ell(\Gamma)}\sin\left(\frac\pi Mj\right)\right)^2$. We believe that the disparity is due to the fact that, in both of these references, the points of the discrete system are equally spaced, whereas in our numerical computation, our points are not equally spaced with respect to $s$.

\subsection{Asymptotics for large $k$}
We now turn our attention to the behavior of the $j$th eigenvalue of $L_k$ for large $k$. In this setting, the potential energy $V$ in $H=-\ddtwo s+V$ is more significant than the kinetic energy $-\ddtwo s$. We restrict our attention to $k\ge2$, in which case the minimum of $V$ occurs at the outermost point of the torus cross-section thanks to the $\frac{k^2-1}{r^2}$ term in $V=\left(\sigma^{-1/2}\right)''\sigma^{1/2}-1+\frac{k^2-1}{r^2}$. As we expect, we see in Figure~\ref{fig:j0} that our wave functions are concentrated near the minimum of the potential $V$. Moreover, as $k$ grows, the potential well becomes steeper, so, as expected, the wave functions become more concentrated as $k$ grows.

As before, we parametrize $\Gamma$ with respect to $g^\sigma$ arc length using the variable $t$ and with respect to Euclidean arc length using the variable $s$. Additionally, we will let the outermost point of the cross-section correspond to $s=t=0$. Near this point, we approximate $V$ with a quadratic potential
\begin{equation*}
  V(s)\approx V(0)+\frac{V''(0)}2s^2,
\end{equation*}
which gives the Hamiltonian for the quantum harmonic oscillator
\begin{equation*}
  H\approx -\ddtwo s+V(0)+\frac{V''(0)}2s^2.
\end{equation*}
The lowest energy state for the quantum harmonic oscillator is slightly larger than the minimum value of the potential energy. More precisely, we have
\begin{equation*}
  \lambda_0\approx V(0)+\sqrt{\frac{V''(0)}2}.
\end{equation*}
More generally,
\begin{equation*}
  \lambda_j\approx V(0)+(2j+1)\sqrt{\frac{V''(0)}2}.
\end{equation*}
We expect this approximation to become better as $k$ grows. Intuitively, the potential well becomes steeper, so the eigenfunction becomes more concentrated near the minimum, so it doesn't ``see'' as well how $V$ differs from its quadratic Taylor approximation.

\begin{figure}
  \centering
  \includegraphics{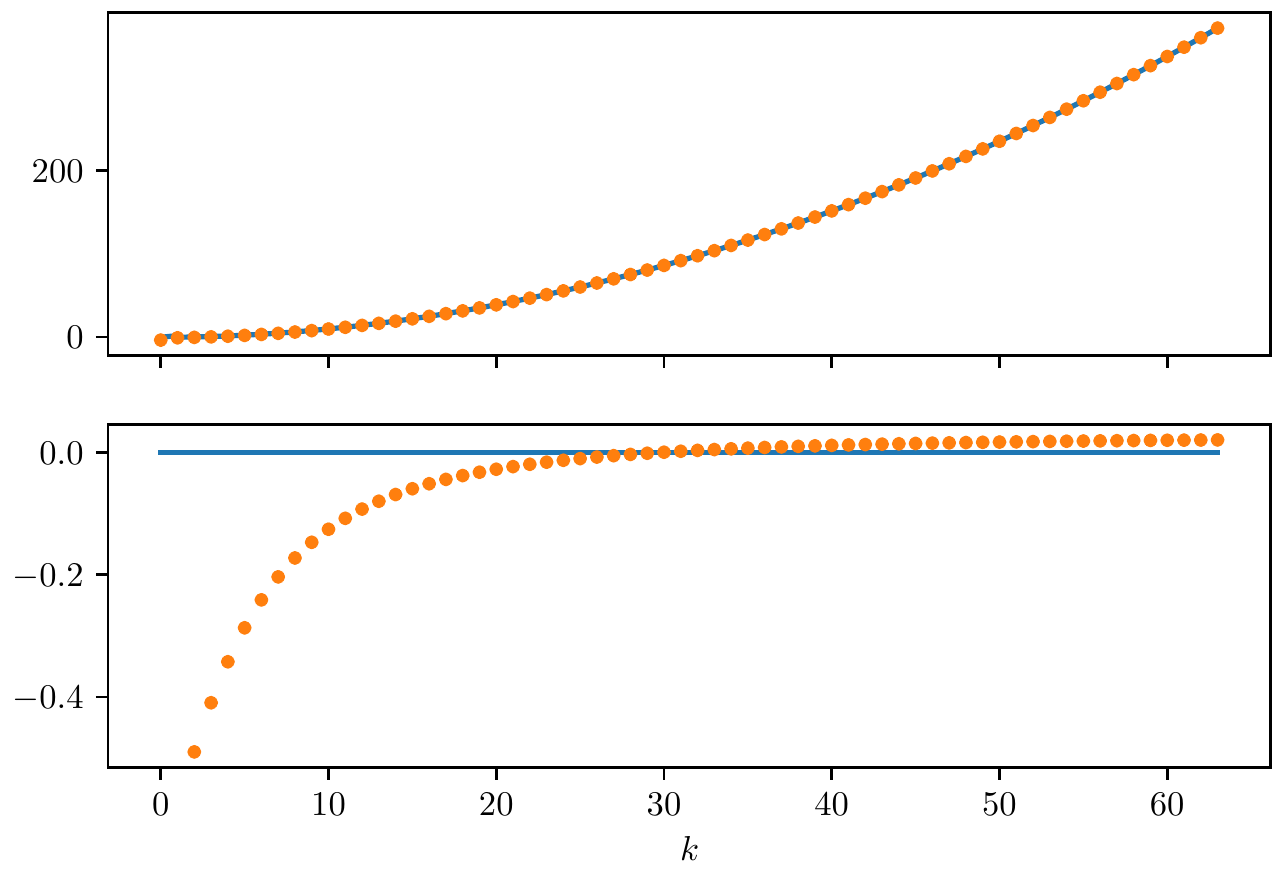}
  \caption{In the upper plot, we plot the eigenvalues $\lambda_0$ (orange dots) of our approximation $L_{k;d}$ to the stability operator $L_k$, along with our asymptotic approximation $V(0)+\sqrt{\frac{V''(0)}2}$ (blue curve). In the lower plot, we plot the difference between the eigenvalues of $L_{k;d}$ and the asymptotic approximation.}
  \label{fig:highk}
\end{figure}

Specializing to $j=0$, we can now look at how well our numerically computed eigenvalues $\lambda_0$ of $L_k$ compare with the estimate $V(0)+\sqrt{\frac{V''(0)}2}$. We plot our findings in Figure~\ref{fig:highk}. We see that, as before, the relative error between our numerically computed eigenvalues and our asymptotic estimate is small. This time, the absolute error is also small, but it still does not tend to zero. This error is sensitive to the number of points $M$, so we suspect that, once again, the behavior of the discrete system is deviating from the behavior of the continuous system. As the region supporting the bulk of the wave function becomes smaller, there become fewer points of the discrete curve in that region.

An additional culprit could be the na\"ive way in which we approximated $V(0)+\sqrt{\frac{V''(0)}2}$. We simply approximated $V$ and its derivatives using finite differences as above, and then evaluated them at the point on the discrete curve where $V$ attains its minimum. In principle, however, we could use either the differential equations defining $\Gamma$ or the algebraic equations defining $\Gamma_d$ to find a formula for $V(0)+\sqrt{\frac{V''(0)}2}$ solely in terms of the maximum value of $r$ along the curve. Doing so would not involve finite differences, so we would expect to get a more accurate estimate.

\end{document}